\documentclass[reqno]{amsart}

\usepackage{amsxtra}
\usepackage{amsopn}
\usepackage{amsmath,amsthm,amssymb,mathbbol}
\usepackage{amscd}
\usepackage{color}
\usepackage{amsfonts}
\usepackage{latexsym}
\usepackage{verbatim}

\theoremstyle{plain}
\newtheorem{theorem}{Theorem}[section]
\newtheorem{lemma}[theorem]{Lemma}
\newtheorem{prop}[theorem]{Proposition}
\newtheorem{cor}[theorem]{Corollary}

\numberwithin{equation}{section}

\theoremstyle{definition}
\newtheorem{definition}[theorem]{Definition}

\newtheorem{rem}[theorem]{Remark}
\newtheorem{ex}[theorem]{Example}

\newcommand\C{{\mathbb C}}

\newcommand\R{{\mathbb R}}

\newcommand{\kset}{{\mathbb k}}

\newcommand{\la}{\lambda}
\newcommand{\La}{\Lambda}
\newcommand{\Ag}{\mathrm{Aut\,}\mathfrak{g}}
\newcommand{\Ai}{{(\mathrm{Aut\,}\mathfrak{g})}^{\circ}}
\newcommand{\Span}{\mathrm{span}}
\newcommand{\Char}{\mathrm{char}}

\newcommand{\ad}{\mathrm{ad}}
\newcommand{\Ric}{\mathrm{Ric}}
\newcommand{\tr}{\mathrm{tr}\,}
\newcommand{\rank}{\mathrm{rank}}

\let\emptyset\varnothing

\begin{document}

\title{Solvable Lie algebras and graphs}
\date{March 27, 2017}

\author[G.~Grantcharov]{Gueo Grantcharov}
\address{G.~Grantcharov, Department of Mathematics and Statistics, Florida International University, Miami, FL 33199, USA}
\email{grantchg@fiu.edu}
 \thanks{The first author is partially supported by Simons Foundation Grant \#246184.}

\author[V.~Grantcharov]{Vladimir Grantcharov}
\address{V.~Grantcharov, School of Computer Science, Georgia Institute of Technology,
Atlanta, GA 30332--0765, USA}
\email{vlad321@gatech.edu}

\author[P.~Iliev]{Plamen~Iliev}
\address{P.~Iliev, School of Mathematics, Georgia Institute of Technology,
Atlanta, GA 30332--0160, USA}
\email{iliev@math.gatech.edu}
\thanks{The third author is partially supported by Simons Foundation Grant \#280940.}

\begin{abstract}
We define a solvable extension of the graph $2$-step nilpotent Lie algebras of \cite{DM} by adding elements corresponding to the $3$-cliques of the graph. We study some of their basic properties and we prove that two such Lie algebras are isomorphic if and only if their graphs are isomorphic.
We also briefly discuss some metric properties, providing examples of homogeneous spaces with nonpositive curvature operator and solvsolitons.
\end{abstract}

\maketitle

\section{Introduction}\label{se1}

Attaching algebraic objects to combinatorial structures is a well-known tool to study properties both on the combinatorial, and on the algebraic side. Often it also has applications in geometry. One such example is \cite{DM}, where Dani and Mainkar defined a $2$-step nilpotent Lie algebra associated to any simple graph. Later, Mainkar \cite{M} showed that such algebras are isomorphic if and only if the corresponding graphs are isomorphic. The construction has been used to study various geometric properties of the Lie algebras of graph type and, in particular, the existence of symplectic structures \cite{PT}, possible nilsolitons and Einstein solvmanifolds defined by solvable extensions \cite{LW1, L2, Laf}. Often additional structures on graphs lead to similar constructions - in \cite{R} a 3-step nilpotent metric Lie algebra has been associated to  Schreier graph. Recently \cite{PS}, the uniform nilpotent Lie algebras, which have extensions as Einstein solvmanifolds, have been encoded via a class of regular edge-colored graphs called uniformly colored graphs.

In this paper we initiate a study of solvable Lie algebras associated to another structure on a graph - the set of its $3$-cliques (or $3$-cycles). The definitions and constructions extend easily to $k$-cliques for any $k\geq 3$ and thus are related to one of the well-known NP-complete problems - the description of $k$-cliques in a graph. For a better clarity however, we focus on the case $k=3$. This case also provides most of the examples relevant to the metric properties of the corresponding Lie group studied in Section~\ref{se4}.

The solvable Lie algebras constructed in the present paper are extensions of the nilpotent graph Lie algebras of \cite{DM} obtained by adding also elements corresponding to the $3$-cliques with diagonal adjoint endomorphisms. In Section~\ref{se2} we collect some of the basic properties of these algebras and, in particular, we characterize the center and the nilradical. Our main result (Section~\ref{se3}) is that two such algebras are isomorphic if and only if their defining graphs are isomorphic. The proof uses properties of maximal tori of the automorphism group, but its construction is slightly different from the one in the nilpotent case \cite{M}. One of the main difficulties here is that the toral group generated by the automorphisms acting diagonally on the standard basis is no longer maximal and thus requires a more detailed analysis. In the last section, we describe some metric properties for solvable $3$-clique graph algebras for which every vertex is adjacent to a $3$-clique. They provide examples of homogeneous spaces of nonpositive curvature operator. The commutator of such algebras is the nilpotent algebra from \cite{DM} of the graph. Using this fact we also show that they define solvsolitons if and only if their commutator defines a nilsoliton, and that the natural basis we choose is stably Ricci diagonal. Finally we mention that the simply connected Lie groups of such algebras, endowed with left-invariant metrics, are isometric if and only if the corresponding graphs are.

\section{Cliques and solvable Lie algebras}\label{se2}

In this section we define a Lie algebra associated to a graph, which extends the definitions in \cite{DM, M}.
Recall that for a Lie algebra $\mathfrak{g}$ we have two descending series of subalgebras: $\mathfrak{g}^k = [\mathfrak{g},\mathfrak{g}^{k-1}]$ and $\mathfrak{g}^{(k)} = [\mathfrak{g}^{(k-1)},\mathfrak{g}^{(k-1)}]$, where $\mathfrak{g}^1 = \mathfrak{g}^{(1)}=\mathfrak{g}'=[\mathfrak{g},\mathfrak{g}]$. The Lie algebra $\mathfrak{g}$ is called {\it n-step nilpotent} if $\mathfrak{g}^{n} = 0$ but $\mathfrak{g}^{n-1}\neq 0$, and $\mathfrak{g}$ is called {\it n-step solvable} if $\mathfrak{g}^{(n)}=0$, $\mathfrak{g}^{(n-1)}\neq 0$.

Let $\kset$ be an arbitrary field.
For a simple finite graph $\mathcal{G}$ with a set of vertices $V(\mathcal{G})$ numbered by $1,2,\dots,n$ and edges $E(\mathcal{G})$ denoted by $(ij)$, we assign a vector $e_i$ to each of vertex and we  define the spaces
$$V=\Span_{\kset}\{e_i| i\in V(\mathcal{G}) \}\qquad {\text{ and }}\qquad W=\Span_{\kset}\{e_i \wedge e_j| (ij) \in E(\mathcal{G})\}\subset \Lambda^2(V).$$
Let $C(\mathcal{G})$ be the set of $3$-cliques $(ijk)$, $i<j<k,$ of the graph $\mathcal{G}$ and consider the space $U$ defined by
$$U=\Span_{\kset}\{e_{ijk}=e_i\odot e_j\odot e_k \in S^3(V)| i<j<k, (ijk) \in C(\mathcal{G})\}.$$
To simplify the notations in rest of the paper, we omit the explicit dependence on the field $\kset$, unless we need to consider spaces and algebras with respect to different fields.

\begin{definition}\label{def1}
Define nontrivial brackets between elements of the set $\{e_{i}\}_{i\in V(\mathcal{G})}\cup \{e_{i}\wedge e_{j}\}_{(ij)\in E(\mathcal{G})}\cup \{e_{ijk}\}_{(ijk)\in C(\mathcal{G})}$ as follows:
\begin{align*}
&[e_i, e_j] = e_i \wedge e_j, && \text{if } (ij)\in E(\mathcal{G}),\\
&[e_a, e_{ijk}] = e_a, &&\text{if }  a \in \{i,j,k\},\\
&[e_a \wedge e_b, e_{ijk}] = 2(e_a \wedge e_b),  &&\text{if }  \{a,b\} \subset \{i,j,k\},\\
&[e_a \wedge e_b, e_{ijk}] = e_a \wedge e_b,  && \text{if exactly one of  $a$ and $b$ is in }\{i, j, k\},
\end{align*}
and extend it to a Lie bracket  on $\mathfrak{g} = V\oplus W\oplus U$ by linearity.
\end{definition}

\begin{rem}
Note that the space $V\oplus W$ is closed under this bracket and corresponds to the $2$-step nilpotent Lie algebra defined in \cite{DM, M} where only the vertices and edges are considered. The space $U$ is a subspace of $\Span_{\kset}\{e_{ijk}\in S^3(V)| i<j<k\}$, which is invariant under the action of the symmetric group $S_n$ of all permutations, but is not an irreducible representation when its dimension is greater than $1$.
\end{rem}
With the above definition we can prove the following theorem.
\begin{theorem}\label{th2.3}
The vector space $\mathfrak{g}$ with the bracket $[\cdot,\cdot]$ from Definition \ref{def1} is a  Lie algebra.
\end{theorem}

\begin{proof} We need only to check that the Jacobi identity holds for elements from the basis $\{e_{i}\}_{i\in V(\mathcal{G})}\cup \{e_{i}\wedge e_{j}\}_{(ij)\in E(\mathcal{G})}\cup \{e_{ijk}\}_{(ijk)\in C(\mathcal{G})}$. Since $V\oplus W$ is a Lie algebra  and brackets between elements in $U$ are always $0$, it is sufficient to check the Jacobi identity when one or two of the elements correspond to $3$-cliques. \\
{\em Case 1.} One $3$-clique $e_{klm}$ and two elements from $V\oplus W$. Note that if the elements from $V\oplus W$ correspond to a vertex and an edge, or two edges, all terms vanish. Thus we only need to consider the Jacobi identity for two vertices $e_{i}$, $e_j$ and the $3$-clique $e_{klm}$ and we can have nonzero terms only if $(ij)$ is an edge and $\{i,j\}\cap\{k,l,m\}\neq \emptyset$. There are two possible sub-cases:\\
1.1 If $\{i,j\}\subset \{k,l,m\}$, say $k=i$ and $l=j$, and we have
\begin{align*}
 [e_i, [e_j, e_{ijm}]] + [e_j, [e_{ijm}, e_i]] + [e_{ijm}, [e_i, e_j]] &= [e_i, e_j] + [e_j, -e_i] + [e_{ijm}, e_i \wedge e_j]\\
 &= e_i \wedge e_j + e_i \wedge e_j  -2(e_i \wedge e_j) = 0.
 \end{align*}
1.2 If $\{i,j\}\cap\{k,l,m\}$ consists of one element, say $k=i$, but $j\neq l, m$. Then
\begin{align*}
[e_i, [e_j, e_{ilm}]] + [e_j, [e_{ilm}, e_i]] + [e_{ilm}, [e_i, e_j]] &= [e_i, 0] + [e_j, -e_i] + [e_{ilm}, e_i \wedge e_j]\\
& = 0 + e_i \wedge e_j + (-e_i \wedge e_j) = 0.
\end{align*}
{\em Case 2.} Two $3$-cliques and one element corresponding to a vertex or an edge.\\
2.1 In the case of two $3$-cliques and one vertex: $e_{ijk}, e_{lmn}, e_o$ we may assume  $o = i = l$ (otherwise all terms are $0$) and we have
\begin{align*}
[e_{ijk}, [ e_{imn}, e_i]] + [e_{imn}, [e_i, e_{ijk}]] + [e_i, [e_{ijk}, e_{imn}] &= [e_{ijk}, -e_i] + [e_{imn}, e_i] + [e_i, 0] \\
&= e_i + (-e_i) + 0 = 0.
\end{align*}
2.2 Two $3$-cliques and one edge: $e_{ijk}, e_{lmn}, e_o\wedge e_p$. By definition  $[e_o\wedge e_p,e_{ijk}]=\alpha (e_o\wedge e_p)$ and $[e_o\wedge e_p,e_{lmn}]=\beta (e_o\wedge e_p)$, where $\alpha,\beta \in \{0,1,2\}$, hence
\begin{align*}
&[e_o\wedge e_p, [e_{ijk}, e_{lmn}]] +[e_{ijk}, [e_{lmn}, e_o\wedge e_p]] + [e_{lmn}, [e_o\wedge e_p, e_{ijk}]\\
&\qquad=0-\beta[e_{ijk}, e_o\wedge e_p]+\alpha [e_{lmn}, e_o\wedge e_p]=\alpha\beta (e_o\wedge e_p)-\alpha\beta (e_o\wedge e_p)=0,
\end{align*}
competing the proof.
\end{proof}

In order to describe the descending series $\mathfrak{g}^{k}$ and $\mathfrak{g}^{(k)}$ of subalgebras we introduce the following notations.
For a graph  $\mathcal{G}$  we denote by
$$\Delta = (V_{\Delta}, E_{\Delta})$$
the subgraph of $\mathcal{G}$ with  vertices which belong to some $3$-clique and the edges between them. Thus every vertex and edge of $\Delta$ is included in some $3$-clique. Similarly let
$$ \Gamma = (V_{\Gamma}, E_{\Gamma})$$
be the graph containing the remaining vertices in $\mathcal{G}$: $V_{\Gamma} = V({\mathcal{G}}) \setminus V_{\Delta}$ and the edges between them. Also denote by $E_{\Delta,\Gamma}$ the edges with one vertex in $V_{\Delta}$ and another in $V_{\Gamma}$. In particular we have:
$$ V(\mathcal{G}) = V_{\Delta}\cup V_{\Gamma}\qquad \text{ and }\qquad E({\mathcal{G}}) = E_{\Delta}\cup E_{\Gamma}\cup E_{\Delta,\Gamma}.$$
With these notations, it is straightforward to describe explicitly the descending series $\mathfrak{g}^{k}$ and $\mathfrak{g}^{(k)}$.
\begin{prop}\label{prop2.4}
For the Lie algebra $\mathfrak{g}$ constructed from a graph $\mathcal{G}$ we have:
\begin{itemize}
\item[(a)] $\mathfrak{g'} = \Span \{ e_{\alpha}  | \alpha \in V_{\Delta}\} \oplus
\Span \{e_{\alpha} \wedge e_{\beta} | (\alpha\beta) \in E(\mathcal{G})\}$;
\item[(b)] $\mathfrak{g}^{(2)} = \Span \{e_{\alpha} \wedge e_{\beta} | (\alpha \beta) \in E_{\Delta}\}$;
\item[(c)] $\mathfrak{g}^{(3)}  =  0$;
\item[(d)] $\mathfrak{g}^{k} =
\Span \{ e_{\alpha} | \alpha \in V_{\Delta} \} \oplus \Span \{ e_{\alpha} \wedge e_{\beta} |  (\alpha \beta) \in E_{\Delta}\cup E_{\Delta,\Gamma}\}$ for $k\geq 2$.
\end{itemize}
\end{prop}

In particular, from the above proposition we see that $\mathfrak{g}$ is solvable.
\begin{cor}
The Lie algebra $\mathfrak{g}$ constructed from the graph $\mathcal{G}$ is solvable, obtained as an abelian extension of a nilpotent Lie algebra. 
\end{cor}
Note also that the Lie algebra $\mathfrak{g}$ is not nilpotent if the graph $\mathcal{G}$ has at least one clique.
Next we characterize the center of $\mathfrak{g}$. Since specific linear combinations of elements in $U$ belong to the center, it is convenient to present the conditions in terms of the {\it $3$-clique incidence matrix}.

\begin{definition}
The $3$-clique incidence matrix of a graph is a matrix $A$ with rows corresponding to the vertices and columns corresponding to $3$-cliques, such that
the entry of $A$ in the row corresponding to the vertex $i$ and $3$-clique $(klm)$  is 1 if $i\in \{k,l,m\}$ and 0 otherwise.
\end{definition}
We can think of the matrix $A$ as a linear transformation from $U$ to $V$ and therefore $\ker(A)\subset U$.
Using this definition, we obtain the following description of the center and the nilradical (the maximal nilpotent ideal) of $\mathfrak{g}$.
\begin{prop}\label{ZandNR}
If $\Char(\kset)\neq 2$ then
\begin{itemize}
\item[(i)]  The center $Z(\mathfrak{g})$ of the solvable Lie algebra $\mathfrak{g}$ is:
\begin{equation}
Z(\mathfrak{g})  =  \Span \{e_{\alpha} | {\alpha} \text{ is an  isolated vertex} \}
 \oplus \Span \{e_{\alpha} \wedge e_{\beta} | (\alpha\beta)\in E_{\Gamma}\}
 \oplus \ker(A) \label{fc}
\end{equation}
where $A$ is the $3$-clique incidence matrix of the graph $\mathcal{G}$.
\item[(ii)]  The nilradical $NR(\mathfrak{g})$ is given by
\begin{align}
NR(\mathfrak{g}) = V\oplus W \oplus \ker(A).
\end{align}
\end{itemize}
\end{prop}

\begin{proof}
Note first that if $[l,e_i]=0$ for all vertices $i$, then $[l,e_i\wedge e_j]=0$ for all edges $(ij)$. This follows from the Jacobi identity for the elements $l$, $e_i$ and $e_j$:
\begin{align*}
[e_i, [e_j, l]] + [e_j, [l, e_i]] + [l, e_i\wedge e_j] = 0.
\end{align*}
Thus, $l\in Z(\mathfrak{g})$ if and only if $l$ commutes with the elements in $\mathfrak{g}$ corresponding to all vertices and 3-cliques. To compute these brackets, we introduce more notations:
\begin{align*}
V(i)&=\{\alpha | (i\alpha) \in E(\mathcal{G})\},\\
C(i)&=\{(\alpha\beta) | (i\alpha\beta)\in C(\mathcal{G})\}.
\end{align*}
We can write an element $l\in \mathfrak{g}$ as a linear combination
\begin{equation}\label{ll}
l=\sum_{i\in V(\mathcal{G})}a_ie_i + \sum_{(ij)\in E(\mathcal{G})}b_{ij}(e_{i}\wedge e_{j}) + \sum_{(ijk)\in C(\mathcal{G})}c_{ijk}e_{ijk}.
\end{equation}
We can assume that initially all multi-indices above are ordered in increasing order of the vertices, i.e. $i<j$ in $e_{i}\wedge e_{j}$ and $i<j<k$ in $e_{ijk}$ for any pairs or triples of indices. Then we can extend and think of $b_{ij}$ as a skew-symmetric function of the indices (i.e. $b_{ij}=-b_{ji}$) and interpret $c_{ijk}$ as a symmetric function of the indices $i,j,k$.
The commutator of $l$ with elements corresponding to a vertex and a $3$-clique are given by the following formulas:
\begin{align}
[l,e_i]&=\sum_{j\in V(i)} a_j(e_{j}\wedge e_{i}) -\left(\sum_{(jk)\in C(i)}c_{ijk}\right)e_i,\label{cv}\\
[l,e_{ijk}]&=a_ie_i+a_je_j+a_ke_k + 2b_{ij}(e_{i}\wedge e_{j})+2b_{jk}(e_{j}\wedge e_{k})+2b_{ik}(e_{i}\wedge e_{k})\nonumber\\
&\qquad +\sum_{\begin{subarray}{c} \alpha\in V(i)\\ \alpha\neq j,k\end{subarray}}b_{i\alpha}(e_{i}\wedge e_{\alpha})
+\sum_{\begin{subarray}{c} \alpha\in V(j)\\ \alpha\neq i,k\end{subarray}}b_{j\alpha}(e_{j}\wedge e_{\alpha})
+\sum_{\begin{subarray}{c} \alpha\in V(k)\\ \alpha\neq i,j\end{subarray}}b_{k\alpha}(e_{k}\wedge e_{\alpha}).\label{cc}
\end{align}
Suppose first that $l\in Z(\mathfrak{g})$. From \eqref{cv} it follows that $a_j=0$ for every nonisolated vertex $j\in V(\mathcal{G})$ and therefore
$l_1=\sum_{i\in V(\mathcal{G})}a_ie_i \in  \Span \{e_{\alpha} | {\alpha} \text{ is an isolated vertex} \} $. Similarly, from \eqref{cc} it follows that $b_{ij}=0$ for every $(ij)\in E_{\Delta}\cup E_{\Gamma,\Delta}$ proving that $l_2= \sum_{(ij)\in E(\mathcal{G})}b_{ij}(e_{i}\wedge e_{j})\in \Span \{e_{\alpha} \wedge e_{\beta} | (\alpha\beta) \in E_{\Gamma}\} $. Finally, equation \eqref{cv} shows that
$l_3= \sum_{(ijk)\in C(\mathcal{G})}c_{ijk}e_{ijk} \in\ker(A)$, completing the proof in this direction.

Conversely, from \eqref{cv} and \eqref{cc} it is easy to see that every element in the right-hand side of equation~\eqref{fc} belongs to $Z(\mathfrak{g})$, completing the proof of the first part.

For the second part we note that $V\oplus W\oplus \ker(A)$ is clearly a nilpotent ideal and for every element $l\in U\setminus\ker(A)$, $\ad_l$ has a non-zero diagonal entry. Therefore there is no nilpotent ideal which properly contains $V\oplus W\oplus \ker(A)$.
\end{proof}

\begin{rem}\label{ap}
From equations \eqref{ll} and \eqref{cv} we see that if
\begin{equation*}
l= \sum_{(ijk)\in C(\mathcal{G})}c_{ijk}e_{ijk}\in \ker(A)=Z(\mathfrak{g})\cap U,
\end{equation*}
then
\begin{equation*}
3\sum_{(ijk)\in C(\mathcal{G})}c_{ijk}=\sum_{i\in V_{\Delta}}\sum_{(jk)\in C(i)}c_{ijk}=0.
\end{equation*}
\end{rem}

\begin{rem}\label{remc}
If $\Char(\kset)= 2$, then $Z(\mathfrak{g})$ may contain also elements from $ \Span \{ e_{\alpha} \wedge e_{\beta} |  (\alpha, \beta) \in E_{\Delta}\}$.
\end{rem}

We say that $\mathfrak{g}$  has a  {\em non-trivial center\/} if $\ker(A)\neq\{0\}$.

\begin{ex} Let $\mathfrak{g}$ be the Lie algebra  constructed from the complete graph $K_n$ and $\Char(\kset)\neq 2,3$. Note that  $A$ is an $n \times \frac{n(n-1)(n-2)}{6}$ matrix and it is easy to see that $\mathfrak{g}$ has a trivial center for every $n\leq 4$. When $n\geq 5$,  $\mathfrak{g}$ has a non-trivial center, since $x=e_{123}-e_{134}+e_{145}-e_{125}\in Z(\mathfrak{g})$. The automorphism group of $K_n$ is $S_n$ and $\ker(A)$ is an $S_n$-invariant subspace of $U$. Since the sum of all $e_{ijk}$ is a fixed vector under this action, Remark \ref{ap} corresponds to the fact that in $U$ non-intersecting invariant subspaces are orthogonal under the natural scalar product. A quick check shows that the elements $\{x, (12)x,(34)x,(12)(45)x, (143)x\}$ are linearly independent and thus
$$Z(\mathfrak{g})=\ker(A)=\Span_{\kset}\{x, (12)x,(34)x,(12)(45)x, (143)x\} \quad \text{ when }n=5.$$
More generally, we can show that $\rank(A)=n$ for $n\geq 5$ and therefore
$$\dim(Z(\mathfrak{g}))=\dim(\ker(A))=\frac{n(n-1)(n-2)}{6}-n\text{ for every } n\geq 5.$$
To see this, consider the $n\times n$ submatrix $\hat{A}$ of $A$, obtained by taking all rows and the columns corresponding to the $3$-cliques $(123)$, $(124)$, $(134)$, $(234)$, $(12j)$ for $j=5,6,\dots,n$. Note that for $j\geq 5$ the $j$-th row has only one nonzero entry in the column corresponding to the $3$-clique $(12j)$. Thus, $\hat{A}$ will be nonsingular if and only if the submatrix $\hat{A}_4$, obtained from $\hat{A}$ by taking the first $4$ rows and the columns corresponding to the $3$-cliques $(123)$, $(124)$, $(134)$, $(234)$ is nonsingular. Since
$$\hat{A}_4=\left(\begin{matrix} 1 & 1 & 1 & 0\\  1 & 1 & 0 & 1\\  1 & 0 & 1 & 1\\  0 & 1 & 1 & 1 \end{matrix}\right)\qquad \text{and }\qquad \det(\hat{A}_4)=-3,$$
we conclude that  $\rank(A)=n$.
\end{ex}

\begin{rem}\label{rem-k-clique}
It is known that the description of $k$-cliques in a graph is NP complete - existence or non-existence of polynomial time algorithm is equivalent to P versus NP problem. We can extend the constructions to define a Lie algebra associated to the set of all $k$-cliques for fixed $k\geq 3$ as follows. Let $C_k(\mathcal{G})$ denote the set of all $k$-cliques $(i_1i_2,\dots i_k)$, $i_1<i_2<\cdots i_k$ of the graph $\mathcal{G}$ and let $U_k$ be the space defined by
$$U_k=\Span_{\kset}\{e_{i_1i_2\dots i_k}=e_{i_1}\odot e_{i_2}\odot\cdots\odot e_{i_k} \in S^{k}(V)| i_1<i_2<\cdots< i_k,\, (i_1i_2\dots i_k) \in C_k(\mathcal{G})\}.$$
Similarly to Definition~\ref{def1}, we can define $\mathfrak{g}(k) = V\oplus W\oplus U_k$ with nontrivial brackets
\begin{align*}
&[e_i, e_j] = e_i \wedge e_j, && \text{if } (ij)\in E(\mathcal{G}),\\
&[e_a, e_{i_1,\dots,i_k}] = e_a, &&\text{if }  a \in \{i_1,\dots,i_k\},\\
&[e_a \wedge e_b, e_{i_1\dots i_k}] = 2(e_a \wedge e_b),  &&\text{if }  \{a,b\} \subset  \{i_1,\dots,i_k\},\\
&[e_a \wedge e_b, e_{i_1\dots i_k}] = e_a \wedge e_b,  && \text{if exactly one of  $a$ and $b$ is in } \{i_1,\dots,i_k\}.
\end{align*}
The proofs of Theorem~\ref{th2.3} and Proposition~\ref{prop2.4} can be easily extended to show that $\mathfrak{g}(k) $ is a solvable Lie algebra and its descending series, center and nilradical can be described by analogous extensions of the constructions above.
\end{rem}

\begin{rem}\label{rem-param}
The constructions can also be generalized by considering weights $\omega_i\in\kset$ on the vertices and by replacing the relations in Definition~\ref{def1} with
\begin{align*}
&[e_i, e_j] = e_i \wedge e_j, && \text{if } (ij)\in E(\mathcal{G}),\\
&[e_a, e_{ijk}] = \omega_a e_a, &&\text{if }  a \in \{i,j,k\},\\
&[e_a \wedge e_b, e_{ijk}] = \left(\sum_{s\in \{a,b\} \cap \{i,j,k\}}\omega_s\right)(e_a \wedge e_b),  &&\text{if }  \{a,b\} \cap \{i,j,k\}\neq \emptyset.
\end{align*}
The proof of Theorem~\ref{th2.3} can be easily extended to this case and  Proposition~\ref{prop2.4} still holds if we assume in addition that 
\begin{equation}\label{cond1}
\omega_i\neq 0 \qquad\text{ for all }\qquad i\in V(\mathcal{G}).
\end{equation}
Finally, if \eqref{cond1} holds and 
\begin{equation}\label{cond2}
\omega_i+\omega_j\neq 0 \qquad\text{ for all }\qquad (ij)\in E(\mathcal{G}),
\end{equation}
then Proposition~\ref{ZandNR} describes the center and the nilradical of the corresponding Lie algebra (cf. Remark~\ref{remc}). In particular, if $\kset=\C$ or $\R$ and if $\omega_i>0$ for all vertices $i$, then equations \eqref{cond1}-\eqref{cond2} are automatically satisfied and all statements established in this section apply for the weighted construction. We thank an anonymous referee for suggesting this weighted generalization of Definition~\ref{def1}. 
\end{rem}

\section{Functorial properties}\label{se3}

\subsection{Preliminary results} In this subsection we assume that the underlying field $\kset$ is algebraically closed and we set $\kset^{*}=\kset\setminus\{0\}$. We 
collect below several preliminary facts needed for the proof of the main theorem.

Let $\Ag$ be the group of Lie algebra automorphisms of $\mathfrak{g}$, and let $\Ai$ be the identity component of $\Ag$ in the Zariski topology. For an $n$-tuple
$\Lambda=(\la_1,\dots,\la_n) \in (\kset^{*})^n$ we define an element $g_{\La}\in \Ag$ by
\begin{align*}
g_{\La} (e_i) &=\la_{i} e_i \text{ for every vertex } i,\\
g_{\La}( e_{i}\wedge e_j )&=\la_{i}\la_j (e_{i}\wedge e_j )\text{ for every edge }(ij),\\
g_{\La}( e_{ijk})&=e_{ijk} \text{ for every $3$-clique }(ijk),
\end{align*}
and extended to $\mathfrak{g}$ by linearity. The torus $T_d$ generated by all $g_{\La}$ is contained in $\Ai$ and we denote by $Z_{\Ai}(T_d)$ its centralizer.

Since every algebraically closed field is infinite, it is not hard to show (say, by induction on $n$) that there exists an $n$-tuple $\La\in (\kset^{*})^n$ as above, for which the numbers:
$1$, $\la_i$ for ${i\in V(\mathcal{G})}$, and $\la_i\la_j$ for ${(ij)\in E(\mathcal{G})}$ are distinct.

\begin{lemma} \label{le1}
If $g\in Z_{\Ai}(T_d)$, then $g$ acts diagonally on all $e_i$ and $e_i\wedge e_j$.
\end{lemma}

\begin{proof}
Consider $g_{\La}\in T_d$ such that the numbers $1$, $\la_i$ for ${i\in V(\mathcal{G})}$, and $\la_i\la_j$ for ${(ij)\in E(\mathcal{G})}$ are distinct. For every vertex $i$ we have
$$g_{\La}(e_i)=\la_i e_i \text{ and therefore applying $g$ we find } g_{\La}(g(e_i))=\la_i g(e_i).$$
Since the eigenspace of $\la_i$ is one-dimensional it follows that $g(e_i)=\mu_i e_i$ for some $\mu_i\in\kset^{*}$. Then for every $(ij)\in E(\mathcal{G})$ we have $g(e_i\wedge e_j)=g([e_i,e_j])=\mu_i\mu_j (e_i\wedge e_j)$.
\end{proof}

\begin{lemma} \label{le2}
Suppose that $g\in Z_{\Ai}(T_d)$ is diagonalizable. Then
\begin{itemize}
\item[(i)]  $g(V)=V,$ $g(W)=W$ and $g(U)=U$.
\item[(ii)] For every $3$-clique $(ijk)\in C(\mathcal{G})$ we have $g(e_{ijk})-e_{ijk}\in Z(\mathfrak{g})$.
\item[(iii)] If $g(x)=\mu x$ for some $x\in U$ and $\mu\neq 1$, then $x\in Z(\mathfrak{g})$.
\end{itemize}
\end{lemma}
\begin{proof}
(i) Suppose that $x\in \mathfrak{g}$ is an eigenvector for $g$. We can decompose $x$ as $x=v+w+u$ where $v\in V$, $w\in W$ and $u\in U$.
The main point is to show that if $g(x)=\mu x$, then $g(v)=\mu v$, $g(w)=\mu w$ and $g(u)=\mu u$. Take arbitrary $\la\in\kset^{*}$ and consider the $n$-tuple $\La=(\la,\la,\dots,\la)$. Then $g_{\La}(x)=\la v +\la^2 w+u$ and since $g_{\La}g=gg_{\La}$ we see that
$$\la g(v)+\la^2 g(w) +g(u)=gg_{\La}(x)=g_{\La}g(x)=\mu(\la v +\la^2 w+u).$$
Since the above is true for all $\la\in\kset^{*}$ (and there are infinitely many such $\la$'s), it follows that $g(v)=\mu v$, $g(w)=\mu w$ and $g(u)=\mu u$. \\
(ii) From Lemma~\ref{le1} we know that $g(e_a)=\mu_ae_a$ for every vertex $a$.
By definition,
\begin{equation}\label{v1}
[e_a,e_{ijk}]=\epsilon e_a, \qquad \text{ where $\epsilon=1$ if $a\in \{i,j,k\}$ and $\epsilon=0$ if $a\not\in \{i,j,k\}$.}
\end{equation}
Applying $g$ we see that $[\mu_ae_a,g(e_{ijk})]=\mu_a\epsilon e_a$, hence $[e_a,g(e_{ijk})]=\epsilon e_a$. Subtracting the last equation from \eqref{v1} it follows that $[e_a,g(e_{ijk})-e_{ijk}]=0$.
Similarly, $[e_a\wedge e_b,g(e_{ijk})-e_{ijk}]=0$ for all $(ab)\in E(\mathcal{G})$. Finally, since $g(e_{ijk})\in U$ we see that $g(e_{ijk})-e_{ijk}\in Z(\mathfrak{g})$.\\
(iii) If $x=\sum c_{ijk}e_{ijk}\in U$, then from (ii) it follows that $g(x)=x+z$, where $z\in Z(\mathfrak{g})$. On the other hand $g(x)=\mu x$ and if $\mu\neq 1$, then $x=\frac{1}{\mu-1} z\in Z(\mathfrak{g})$.
\end{proof}

\begin{cor}\label{cor1}
If $g\in Z_{\Ai}(T_d)$ is diagonalizable and $\mu\neq 1$ is a simple eigenvalue of $g$, then the corresponding eigenspace is
\begin{itemize}
\item either $\kset\, e_i$ for a vertex $i\in V(\mathcal{G})$,
\item or $\kset (e_i\wedge e_j)$ for an edge $(ij)\in E(\mathcal{G})$,
\item or a one-dimensional subspace of $Z(\mathfrak{g})$.
\end{itemize}
\end{cor}

\subsection{Isomorphisms between graphs and  Lie algebras}
We are now ready to prove the main theorem in this section for arbitrary (not necessarily algebraically closed) field $\kset$.

\begin{theorem}\label{main}
Two solvable Lie algebras constructed from graphs are isomorphic if and only if the graphs are isomorphic.
\end{theorem}

\begin{proof}
Clearly, by construction, isomorphic graphs lead to isomorphic Lie algebras, so we focus below on the opposite direction.

Note first that if two Lie algebras $\mathfrak{g}$ and $\hat{\mathfrak{g}}$ over a field $\kset$, which is not algebraically closed, are isomorphic, then the natural extensions of these algebras:  $\mathfrak{g}\otimes_{\kset} \overline{\kset}$ and $\hat{\mathfrak{g}}\otimes_{\kset} \overline{\kset}$ to the algebraic closure $\overline{\kset}$ of $\kset$ are also isomorphic. Thus, without any restriction, we can assume that  $\kset$ is algebraically closed. Moreover, using the isomorphism we can identify  $\mathfrak{g}$ and $\hat{\mathfrak{g}}$ and therefore the main point is to show that if a solvable Lie algebra $\mathfrak{g}$ can be constructed from two graphs $\mathcal{G}$ and $\hat{\mathcal{G}}$, then the graphs $\mathcal{G}$ and $\hat{\mathcal{G}}$ are isomorphic.

Let $T_d$ and $\hat{T}_d$ be the tori constructed in the previous subsection from $\mathcal{G}$ and $\hat{\mathcal{G}}$, respectively, and let $T\supset T_d$ and  $\hat{T}\supset \hat{T}_d$ be maximal tori in  $\Ai$. Since $\Ai$ is a connected linear algebraic group, there exists $g\in \Ai$ such that $T=g^{-1}\hat{T}g$. Fix $g_{\La}\in T_d$ such that the numbers $1$, $\la_i$ for ${i\in V(\mathcal{G})}$, and $\la_i\la_j$ for ${(ij)\in E(\mathcal{G})}$ are distinct and let $\hat{g}=g g_{\La}g^{-1}\in \hat{T}$. Note that for every vertex $i\in V(\mathcal{G})$ we have $\hat{g} (g(e_i))=g(g_{\La}(e_i))=\la_i g(e_i)$. Since $\la_i$ is a simple eigenvalue of $\hat{g}\in Z_{\Ai}(\hat{T}_d)$, Corollary~\ref{cor1} implies that $g(e_i)$ is
\begin{itemize}
\item either $\mu_i \hat{e}_{i'}$ for a vertex $i'\in V(\hat{\mathcal{G}})$ and $\mu_i\in\kset^{*}$,
\item or $\mu_{i} (\hat{e}_{i'}\wedge \hat{e}_{j'})$ for an edge $(i'j')\in E(\hat{\mathcal{G}})$ and $\mu_i\in\kset^{*}$,
\item or belongs to $Z(\mathfrak{g})$.
\end{itemize}
We show next that if $i\in V(\mathcal{G})$ is a non-isolated vertex, then $g(e_i)=\mu_{i}\hat{e}_{i'}$ for a non-isolated vertex $i'\in V(\hat{\mathcal{G}})$. Indeed, it is clear that $g(e_i)\not\in Z(\mathfrak{g})$. Moreover, if $(ik)\in E(\mathcal{G})$, then $g(e_k)$ must also be either $\mu_{k}\hat{e}_{k'}$ for a non-isolated vertex $k'\in V(\hat{\mathcal{G}})$ or $\mu_{k} (\hat{e}_{k'}\wedge \hat{e}_{l'})$ for an edge $(k'l')\in E(\hat{\mathcal{G}})$. From Definition~\ref{def1} it follows that $g(e_i\wedge e_k)=[g(e_i),g(e_k)]$ can be nonzero only if $g(e_i)=\mu_{i}\hat{e}_{i'}$ and $g(e_k)=\mu_{k}\hat{e}_{k'}$.

The above argument shows that the number of non-isolated vertices of $\mathcal{G}$ is less or equal than the number of non-isolated vertices of $\hat{\mathcal{G}}$. Exchanging the roles of  $\mathcal{G}$ and $\hat{\mathcal{G}}$ we see that they must contain the same number of non-isolated vertices, and for every non-isolated vertex $i\in V(\mathcal{G})$ there exists a unique  non-isolated vertex $i'\in V(\hat{\mathcal{G}})$ such that  $g(e_i)=\mu_{i}\hat{e}_{i'}$ for some $\mu_i\in\kset^{*}$. Moreover, for non-isolated vertices $i,j\in V(\mathcal{G})$ we have
$$g([e_i,e_j])=\mu_i\mu_j[\hat{e}_{i'},\hat{e}_{j'}].$$
The left-hand side is nonzero if and only if $(ij)$ is an edge in $\mathcal{G}$, while the right-hand side is nonzero if and only if $(i'j')$ is an edge in $\hat{\mathcal{G}}$.

If $\mathcal{G}'$ and $\hat{\mathcal{G}'}$ are the subgraphs obtained by removing the isolated vertices from $\mathcal{G}$ and $\hat{\mathcal{G}}$, respectively, then the map $i\to i'$ defines an isomorphism between the graphs $\mathcal{G}'$ and $\hat{\mathcal{G}'}$. In particular, this implies that $\mathcal{G}$ and $\hat{\mathcal{G}}$ have the same number of non-isolated vertices, edges and $3$-cliques, and therefore they must also have the same number of isolated vertices (by computing the dimension of $\mathfrak{g}$ using  $\mathcal{G}$ and $\hat{\mathcal{G}}$). It is easy to see now that the graphs $\mathcal{G}$ and $\hat{\mathcal{G}}$ are also isomorphic.
\end{proof}

\begin{rem}
The proof of Theorem~\ref{main} can be easily extended to the solvable Lie algebras constructed from $k$-cliques defined in Remark~
\ref{rem-k-clique}. More precisely, for fixed $k\geq 3$, two solvable Lie algebras $\mathfrak{g}(k)$ and  $\hat{\mathfrak{g}}(k)$ constructed from graphs are isomorphic if and only if the graphs are isomorphic. The isomorphisms of the Lie algebras defined in Remark~\ref{rem-param} are more delicate since rescaling even some of the weights may lead to isomorphic Lie algebras.
\end{rem}

\section{Metric properties}\label{se4}

As an application, in this  section we point out some metric properties of the solvable Lie algebras constructed from graphs. The most interesting properties appear on algebras corresponding to graphs with "many" $3$-cliques - such that every vertex belongs to some $3$-clique. Since we are interested in the metric properties, {\it in this section we fix $\mathbb{k}=\mathbb{R}$}.

Recall that a solvable Lie algebra over $\mathbb{R}$ is {\it completely solvable} if all inner derivations have only real eigenvalues (see e.g. \cite{H}).
For a Lie algebra $\mathfrak{g}$ and an element $B\in \mathfrak{g}$ we use $\ad_B$ to denote the adjoint map $\ad_B: \mathfrak{g}\to  \mathfrak{g}$ defined by $\ad_BC=[B,C]$.
If $\mathfrak{g} = \mathfrak{a}\oplus \mathfrak{g}'$ for a vector space $\mathfrak{a}$, then $\mathfrak{g}$ is completely solvable if and only if $\ad_B$ has only real eigenvalues for all $B\in\mathfrak{a}$. We will make use of the following:

\begin{prop}\label{c-solv}
The solvable Lie algebra $\mathfrak{g}$ corresponding to a graph in which every vertex belongs to some $3$-clique
is completely solvable.
\end{prop}

\begin{proof}
We only have to note that the condition leads to $\mathfrak{g}'=V\oplus W$ and all basic elements of $U$ have adjoint endomorphisms acting diagonally on the basis  $\{e_{i}\}_{i\in V(\mathcal{G})}\cup \{e_{i}\wedge e_{j}\}_{(ij)\in E(\mathcal{G})}$ with eigenvalues $0$, $-1$, or $-2$.
\end{proof}

Every solvable Lie algebra has an associated simply connected solvable Lie group and the metric structures of such groups have been studied by several authors. The emphasis is on the study of the left-invariant metrics, which are determined by a scalar product on the Lie algebra. Let $G$ be the simply connected Lie group with Lie algebra $\mathfrak{g}$ constructed from the graph $\mathcal{G}$. Let $g$ be an inner product on $\mathfrak{g}$ which defines a left-invariant Riemannian metric on $G$. Denote by $R(X,Y)Z= \nabla_X\nabla_YZ-\nabla_Y\nabla_XZ - \nabla_{[X,Y]}Z$ the curvature of $g$ where $\nabla$ denotes the Levi-Civita connection. Then $R$ is determined by:

\begin{align*}
g(R(X,Y)Y,X) & = -\frac{3}{4}||[X,Y]||^2 - \frac{1}{2}g(\ad_X^2Y,Y) - \frac{1}{2}g(\ad_Y^2X,X)\\
 &\quad - g(\ad_X^tX, \ad_Y^tY)+ \frac{1}{4}||\ad_X^tY+\ad_Y^tX||^2,
\end{align*}
for $X,Y\in \mathfrak{g}$.

 A metric has nonpositive curvature if $g(R(X,Y)Y,X)\leq 0$ for all vectors $X$ and $Y$. The curvature tensor $g(R(X,Y)Z,T))$ which is obtained by polarization from $g(R(X,Y)Y,X)$ defines also the curvature operator $\mathcal{R}: \Lambda^2(\mathfrak{g}) \rightarrow \Lambda^2(\mathfrak{g})$ as $g(\mathcal{R}(X\wedge Y), Z\wedge T) = g(R(X,Y)Z,T)$. From the basic curvature properties it follows that $\mathcal{R}$ is a symmetric operator and a metric for which $\mathcal{R}$ is negative semi-definite is called a metric with {\it nonpositive curvature operator}. A metric with nonpositive curvature operator has nonpositive curvature, while the converse is not always true.
  It is known that a Riemannian homogeneous space with nonpositive curvature operator is a direct product of a flat space and solvable Lie group with left-invariant metric with some "Iwasawa type" condition \cite{AW1, Wr}. A metric solvable Lie algebra $(\mathfrak{g}, g)$ is said to be of {\it Iwasawa type} if it satisfies (see \cite[p.~302]{H}) the following:
\begin{itemize}
\item[(a)] $\mathfrak{g} = \mathfrak{a}\oplus \mathfrak{g}'$, for an abelian complement $\mathfrak{a}$ orthogonal to $\mathfrak{g}'$.
\item[(b)] All operators $\ad_B$ for $B\in\mathfrak{a}$ are symmetric on $\mathfrak{g}$ with respect to $g$, and nonzero for $B\neq 0$.
\item[(c)] There exists a vector $B^0\in \mathfrak{a}$, such that $\ad_{B^{0}}|_{\mathfrak{g}'}$ is positive definite.
\end{itemize}
According to \cite[Theorem 3.2]{Wr} a metric solvable Lie algebra of Iwasawa type admits a scalar product $g'$ with nonpositive curvature operator.

The Ricci curvature is given by $\Ric(X,Y)=\sum_{i} g(R(X,E_i)E_i,Y)$, for an orthonormal basis  $E_i$ of the tangent space. It
defines a Ricci operator $\mathcal{R}ic(X)$ as $g(\mathcal{R}ic(X),Y) = \Ric(X,Y)$.
If $\tr A$ is the trace of the endomorphism $A:\mathfrak{g}\rightarrow\mathfrak{g}$ and $H_g$ is defined as $g(H_g,Y)=\tr(\ad_Y)$, the formula for the Ricci tensor on the metric solvable Lie algebra $\mathfrak{g}$ satisfying (a) and (b) of the conditions for Iwasawa type above is given by \cite{H}:
\begin{equation}\label{Ric}
\Ric(X,X) = -\frac{1}{2}\tr(\ad_X^2)-\frac{1}{2}\tr (\ad_X\,\ad_X^t)+\frac{1}{4}\sum_{i<j}g([E_i,E_j], X)^2 - g(\ad_{H_g} X,X).
\end{equation}

The metric with the property $\Ric(X,Y)=\lambda g(X,Y)$ for a scalar $\lambda$ is called {\it Einstein}. Such metrics are of significant interest in Riemannian geometry and the only known examples of homogeneous Einstein manifolds with negative $\lambda$ are provided by solvable Lie groups. Although our solvable Lie algebras arising from graphs do not provide new examples of  Einstein metrics, they satisfy certain related conditions. One of them, which appeared in the developments of the Ricci flow is {\it Ricci soliton}. A metric $g$ is called a {\it Ricci soliton} if there is a constant $\lambda$ and a vector field $X$ such that $\Ric = cg+ Lie_Xg$ where $Lie_Xg$ is the Lie derivative of $g$ in direction of $X$, given by $(Lie_Xg)(Y,Z) = X(g(Y,Z))-g([X,Y],Z)-g(Y,[X,Z])$ for arbitrary vector fields $Y,Z$. A Ricci soliton on a solvable (resp. nilpotent) Lie algebra is also called {\it a solvsoliton} (resp. {\it a nilsoliton}). Sometimes the Lie algebras admitting solvsolitons (or nilsoliton) metrics are also called solvsoliton (nilsoliton respectively). It is known that a sufficient condition for a metric $g$ to be a solvsoliton is
$$\mathcal{R}ic(X) = cX + D(X)\qquad \text{ for every }X,$$
where $c$ is some constant, and $D$ is a derivation of the Lie algebra.
Another notion which appears in relation to the Ricci flow on parallelizable manifolds is {\it stable Ricci-diagonal basis} \cite{LW}. A basis  $(X_1, X_2, ..., X_n)$ of vector fields of a Lie algebra $\mathfrak{g}$ is called {\it stably Ricci-diagonal} if the Ricci operator of any diagonal invariant metric $g$ (so $g(X_i, X_j)=0$ for $i\neq j$) in  $\mathfrak{g}$ is diagonal in it.

The main metric properties for algebras arising from graphs with $3$-cliques adjacent to every vertex are given in the following:

\begin{theorem}\label{metric}
Let every vertex of the graph  $\mathcal{G}$ belong to some $3$-clique. Then the simply-connected solvable Lie group $G$ with Lie algebra $\mathfrak{g}$ constructed from $\mathcal{G}$ satisfies:
\begin{enumerate}
\item  There is a metric $g$ such that $(G, g)$ has a nonpositive curvature and is isometric to a product of $G_1\times G_2$ where $G_2$ is an abelian Lie group, corresponding to the center of $\mathfrak{g}$ with a flat metric and $G_1$ is solvable Lie group with Lie algebra $\mathfrak{g}/Z(\mathfrak{g})$. Moreover $G_1$ has a metric with nonpositive curvature operator with no flat factor.
\item The basis $e_i, e_j\wedge e_k, e_{ijk}$  is stably Ricci-diagonal.
\item  If $\mathfrak{g} '$ is a nilsoliton, then $\mathfrak{g}$ is a solvsoliton.
\end{enumerate}
\end{theorem}

\begin{proof}
Take the metric for which the basis corresponding to the vertices, edges and $3$-cliques is orthonormal.
Notice that the condition implies that $\mathfrak{g}'=V \oplus W$ and therefore
$\mathfrak{g} = \mathfrak{a}\oplus\mathfrak{g}'$ where $\mathfrak{a}=U$ is an abelian complement spanned by the vectors corresponding to $3$-cliques. From Proposition~\ref{ZandNR} it follows that $Z(\mathfrak{g})=\ker A $, where $A$ is the $3$-clique incidence matrix and we can decompose $\mathfrak{a}$ as $\mathfrak{a} = \ker A \oplus \overline{\mathfrak{a}}$, with $\overline{\mathfrak{a}}=(\ker A)^{\perp}$. If we set $\mathfrak{g}_1= \mathfrak{g}'\oplus\overline{\mathfrak{a}}$ and $\mathfrak{g}_2 = \ker A$, then $\mathfrak{g} = \mathfrak{g}_1\oplus\mathfrak{g}_2$ is an orthogonal decomposition with corresponding decomposition of simply connected Lie groups $G = G_1\times G_2$. We claim that this decomposition has the properties mentioned in part $(1)$. Clearly, $G_2$ is an abelian Lie group, corresponding to the center  $\mathfrak{g}_2=Z(\mathfrak{g})$ with a flat metric. To check that $G_1$ has nonpositive curvature operator with no flat factor we show  that $\mathfrak{g}_1$ is of Iwasawa type. Note first that $\mathfrak{g}'=[\mathfrak{g},\mathfrak{g}]=[\mathfrak{g}_1,\mathfrak{g}_1]=\mathfrak{g}_1'$. Observe also that for any $B \in \overline{\mathfrak{a}}$, $\ad_B$ is diagonal on the orthonormal basis of $ \mathfrak{g}'$ and therefore symmetric. Moreover, $\ad_B|_{{\mathfrak{g}'}}=0$ if and only if $B\in Z(\mathfrak{g})=\ker A$, hence $\ad_B|_{{\mathfrak{g}'}}\neq 0$ for $0\neq B\in \overline{\mathfrak{a}}$. Finally, from Remark~\ref{ap} we see that
\begin{equation*}
B^{0}=-\sum_{(ijk)\in C(\mathcal{G})}e_{ijk} \in \overline{\mathfrak{a}}.
\end{equation*}
Since  $\ad_{B^0}|_{\mathfrak{g}'}$ is positive definite, $\mathfrak{g}_1$ is of Iwasawa type and the proof of (1) follows from \cite{Wr}.

To prove parts (2) and (3), note that  we only have to check them for $\mathfrak{g}_1$. The Ricci tensor satisfies the following (see equation~\eqref{Ric} or \cite[Lemma~4.4]{H}):
 \begin{align*}
&\Ric(B,C)= -\tr(\ad_B\,\ad_C)\\
& \Ric(B, X)= 0\\
&\Ric(X,Y) = \Ric_{\mathfrak{g}'}(X,Y) - g([H_g, X], Y)
\end{align*}
for $B,C \in\overline{\mathfrak{a}}$ and $X,Y \in \mathfrak{g}'$. From this and the fact that $\ad_{H_g}$ is diagonal part (2) follows. Then (3) follows from \cite[Theorem 4.3]{L}.
\end{proof}

\begin{ex}
Consider again the complete graph $K_n$. The algebra $\mathfrak{g}'$ is the $2$-step nilpotent Lie algebra of \cite{DM} and is the only one of type $(m,n)$ with $m = \frac{1}{2}n(n-1)$. Recall that a $2$-step nilpotent algebra is of type $(m,n)$ if its dimension is $m+n$ and its commutator has dimension $m$. The algebra $\mathfrak{g}'$ is an Einstein nilradical which is equivalent to being nilsoliton - see \cite{L2} for example.
\end{ex}

A further study of nilpotent graph Lie algebras which are nilsolitons and their possible extensions to solvsolitons is presented in \cite{LW1, Laf}. The existence of nilsoliton metric is equivalent to the existence of a positive solution of a linear system with variables corresponding to weights of the edges \cite{LW1}. A graph for which such solution exists is called {\it positive}. It is noted in \cite{LW1} that a regular graph (in which all vertices have the same degree) is positive. In particular, apart from the complete graph $K_n$ in the example above, the dodecahedron graph also satisfies the conditions of Theorem \ref{metric} and provides an example of a solvsoliton. In \cite{Laf} a  necessary condition for existence of solutions for this system for a graph $\mathcal{G}$ is related to a simpler graph called the {\it coherence graph} of $\mathcal{G}$. To describe it, first call two vertices {\it similar}, if they have the same neighbors. Then the vertices of the coherence graph are the coherent components (equivalence classes) of vertices in $\mathcal{G}$ and between two components there is an edge, if there is at least one edge connecting the respective sets of vertices. Since a graph containing only one component is either the complete graph, or graph with no edges, the graphs satisfying Theorem \ref{metric} are precisely the ones for which all components are complete graphs. A solution for the system of equations for the nilsoliton reduces to a solution of a lower order system for weights of the edges of the coherence graph. In this way we can see that there are more examples of solvsolitons arising from graphs with 3-cliques mentioned in Section 5 of \cite{Laf}.

Finally we consider the question whether the graphs corresponding to  two isometric solvable Lie algebras are isomorphic. We note that for general metric solvable Lie algebras of nonpositive curvature this is not true \cite{Al2, AW2}. However in our case we have the following:

\begin{theorem}
 Let $\mathfrak{g}$ and $\hat{\mathfrak{g}}$ be two solvable Lie algebras arising from graphs, for which every vertex belongs to some $3$-clique. Assume that there are left-invariant metrics $g$ and $\hat{g}$ which are isometric, meaning that there is an isometry between the corresponding simply connected Lie groups. Then the Lie algebras as well as  the corresponding graphs are isomorphic.
\end{theorem}

\begin{proof}

By a theorem of Alekseevskii \cite{Al1}, if two metric completely solvable Lie algebras are isometric, then they are isomorphic. By Proposition \ref{c-solv}$,\mathfrak{g}$ and $\hat{\mathfrak{g}}$ are completely solvable, so they are isomorphic. Then from Theorem \ref{main} the graphs of $\mathfrak{g}$ and $\hat{\mathfrak{g}}$ are isomorphic.
\end{proof}

{\bf Acknowledgements.} The first author would like to acknowledge the hospitality of Max Plank Institute for Mathematics in Bonn and Institute of Mathematics and Informatics of the Bulgarian Academy of Sciences in Sofia, where part of the work on this project has been done.

\end{document}